\documentclass{article}
\usepackage[a4paper, total={5.9in, 8.4in}]{geometry}
\usepackage{asymptote} 
\usepackage{enumitem}
\usepackage{fancyhdr}
\usepackage{xparse}
\usepackage{amsmath}
\usepackage{amssymb}
\usepackage{amsthm}
\usepackage{hyperref}
\usepackage[nameinlink]{cleveref}
\usepackage{listings}
\usepackage{bbm}
\usepackage{tikz}
\usepackage{xcolor}
\usepackage{float}
\usepackage[sc,medium,center]{titlesec}
\usepackage[normalem]{ulem}

\usetikzlibrary{patterns}

\newtheoremstyle{mydefstyle}{3pt}{3pt}{\normalfont}{}{\bfseries}{.}{0.5em}{}
\theoremstyle{mydefstyle}
\newtheorem{theorem}{Theorem}
\newtheorem*{remark}{Remark}
\newtheorem*{definition}{Definition}
\newtheorem{lemma}[theorem]{Lemma}
\newtheorem{corollary}[theorem]{Corollary}
\newtheorem*{lemma*}{Lemma}
\newtheorem*{theorem*}{Theorem}

\newcommand{\zz}[0]{\mathbb Z}
\begin{document}
\begin{center}
{\Large\bf Collinear Interior Lattice Points in Triangles Satisfying $B(T)\in\{4,5\}$}\\
\vskip 20pt
{\bf Jonathan Sakunkoo}\\
{\small\itshape Mathematical Institute, University of Oxford, Oxford, UK}\\
{\tt jonathan.sakunkoo@cs.ox.ac.uk}\\ 
\vskip 10pt
{\bf Annabella Sakunkoo}\\
{\small\itshape Stanford University Online High School, Redwood City, CA, USA}\\ 
{\tt apianist@ohs.stanford.edu}\\ 
\vskip 10pt
{\bf Dana Paquin}\\
{\small\itshape Stanford University Online High School, Redwood City, CA, USA}\\
{\small\itshape Department of Mathematics, California Polytechnic State University, San Luis Obispo, CA, USA}\\ 
{\tt dpaquin@stanford.edu}\\ 
\end{center}
\date{}

\thispagestyle{fancy}

\pagestyle{fancy}
\fancyhead{}
\fancyhead[R]{\thepage}
\fancyhead[L]{Sakunkoo et al.}
\fancyfoot{}


\begin{abstract}
A positive integer $k$ is called $Bn$-collinear if at least one lattice triangle with $n$ boundary points ($B(T)=n$) and $k$ interior points exists, and every such triangle has all of its interior points collinear. Building on prior work on $B(T)=3$, we completely classify the $B4$- and $B5$-collinear integers. Using canonical lattice representations together with arithmetic properties of Alder’s generalized totient function $g(k)$, we prove that the only $B4$-collinear integers are $k \in \{1, 2, 5\}$. Furthermore, we show that \textit{no} integer is $B5$-collinear. This establishes a structural contrast: while three and four boundary lattice points exhibit some collinearity constraints, five boundary points disrupt the pattern. 

\end{abstract}

\section{Introduction}



The study of lattice point geometry lies at a rich intersection of geometry, number theory, and combinatorics. By constraining geometric shapes to the discrete lattice $\mathbb{Z}^2$, continuous concepts like area transform into counting problems. Classical results such as Pick’s Theorem\cite{pick1899geometrisches} relate the area of a lattice polygon to its numbers of boundary and interior lattice points. While the total count of interior lattice points has been extensively studied, considerably less is known about the geometric configurations of the interior lattice points. 

Collinearity is a fundamental geometric constraint on a configuration of points. Determining when combinatorial data, such as the numbers of boundary and interior lattice points, forces such rigidity is a fascinating problem in discrete geometry. We investigate when the interior lattice points of a lattice triangle are necessarily collinear.
Li--Paquin\cite{li2025lattice} showed that, for lattice triangles with three boundary lattice points, the interior points are forced to be collinear exactly when their number is 1, 2, 4, or 7. This result raises the natural problem of classifying the integers for larger boundary counts. 

In this paper, we completely resolve the next two cases of four and five boundary lattice points, B(T)=4 and B(T)=5, and show that the behavior changes sharply between them. We prove that for a triangle with four boundary points, the interior points are necessarily collinear exactly when there are 1, 2, or 5 of them; for five boundary lattice points, however, no number of interior points forces collinearity. These results suggest that the interaction between boundary structure and interior-point geometry may be more subtle than previously understood.

This paper is organized as follows. In Section 2, we introduce a canonical framework for lattice triangles with four boundary lattice points and establish the geometric facts underlying the classification. In Section 3, we develop a non-collinearity criterion, introduce the required arithmetic tools, including a special case of Alder’s generalized totient function, and classify the interior-point counts that force collinearity when there are four boundary lattice points. Section 4 assembles these results into the proof of the main theorem for four boundary lattice points. Finally, in Section 5, we turn to five boundary lattice points and prove that no number of interior lattice points forces collinearity.

\begin{definition}
A \textit{lattice polygon} is a polygon whose vertices have integral coordinates, i.e., are all in $\zz^2$. 
\end{definition}

\begin{theorem*}[Pick's Theorem]\cite{pick1899geometrisches}
For any simple lattice polygon with $b$ boundary points and $k$ interior points, its area is \[k+\frac b2-1.\] 
\end{theorem*}

\begin{definition}
If for all lattice triangles with exactly $n$ boundary points and $k$ interior points, the triangle's $k$ interior points are all collinear, then such a positive integer $k$ is called \textit{$Bn$-collinear}, provided such a triangle exists. 
\end{definition}

Li--Paquin\cite{li2025lattice} studied $B3$-collinear integers and showed that the only $B3$-collinear integers are $1$, $2$, $4$, and $7$. A natural next step after $3$ boundary points is to consider lattice triangles with $4$ boundary points: $B4$-collinear integers. The following theorems are our main results: 

\begin{theorem}
\label{main:label}
The only $B4$-collinear integers are $1$, $2$, and $5$. 
\end{theorem}

\begin{theorem}
\label{main5:label}
There are no $B5$-collinear integers. 
\end{theorem}

We begin with the $B4$-collinear integers. 

First, for any number of interior points $k$, the triangle with vertices $(0,0)$, $(2,0)$, and $(1,k+1)$ has $4$ boundary points and $k$ interior points, so we need not worry about the existence clause in our definition for $4$ boundary points: it is always satisfied. 

We note that $1$ and $2$ are trivially $B4$-collinear: a singleton set is always collinear, and a set of two points is also always collinear. 

Additionally, the triangles in \Cref{threeandfour:sample} show that $3$ and $4$ are not $B4$-collinear. 

Our proof adapts the canonical-form framework of \cite{li2025lattice}. For the $B(T)=4$ case, we introduce new arithmetic arguments based on Alder's generalized totient function. 




\begin{figure}[H]
	\centering
	\begin{tikzpicture}[scale=0.6]
	\draw[step=1.0,black,very thin] (-2,-1) grid (5,11);
	\draw[black, thick] (0,-1) -- (0,11);
	\draw[black, thick] (-2,0) -- (5,0);
	\draw[black, very thick] (0,0) -- (1,0) -- (3,8) -- (0,0);
	\filldraw[blue] (0,0) circle (3pt) node[anchor=north east]{$(0,0)$};
	\filldraw[blue] (1,0) circle (3pt) node[anchor=north west]{$(1,0)$};
	\filldraw[blue] (3,8) circle (3pt) node[anchor=south east]{$(3,8)$};
	\filldraw[red] (1,1) circle (3pt);
	\filldraw[red] (1,2) circle (3pt);
	\filldraw[red] (2,5) circle (3pt);
	\filldraw[blue] (2,4) circle (3pt) node[anchor=north west]{$(2,4)$};
	\end{tikzpicture}\qquad
	\begin{tikzpicture}[scale=0.6]
	\draw[step=1.0,black,very thin] (-2,-1) grid (5,11);
	\draw[black, thick] (0,-1) -- (0,11);
	\draw[black, thick] (-2,0) -- (5,0);
	\draw[black, very thick] (0,0) -- (1,0) -- (3,10) -- (0,0);
	\filldraw[blue] (0,0) circle (3pt) node[anchor=north east]{$(0,0)$};
	\filldraw[blue] (1,0) circle (3pt) node[anchor=north west]{$(1,0)$};
	\filldraw[blue] (3,10) circle (3pt) node[anchor=south east]{$(3,10)$};
	\filldraw[red] (1,1) circle (3pt);
	\filldraw[red] (1,2) circle (3pt);
	\filldraw[red] (1,3) circle (3pt);
	\filldraw[red] (2,6) circle (3pt);
	\filldraw[blue] (2,5) circle (3pt) node[anchor=north west]{$(2,5)$};
	\end{tikzpicture}
	\caption{Lattice triangles with $4$ boundary points and $\{3,4\}$ interior points, respectively, which are not collinear. Interior points in red, boundary points in blue.}
	\label{threeandfour:sample}
\end{figure}
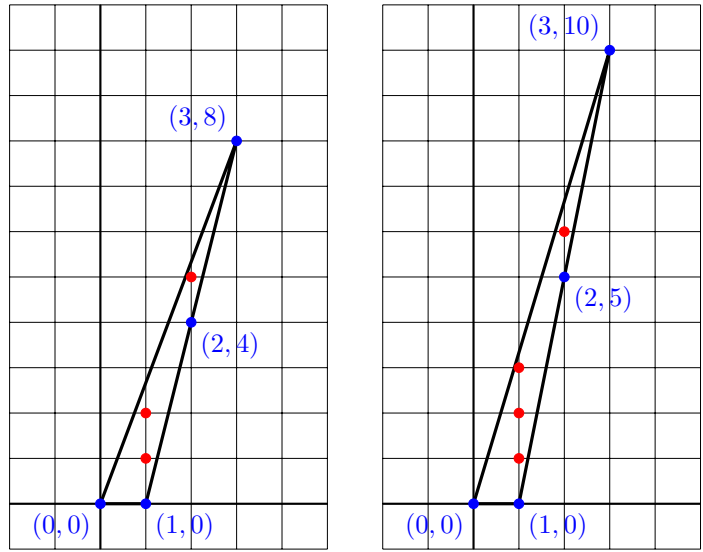


\section{Canonical Form ($B4$)} \label{canonical_section}



Adapting the canonical form established by Li--Paquin\cite{li2025lattice}, all lattice triangles with exactly $4$ boundary points correspond to a ``canonical'' lattice triangle with vertices at the origin $O=(0,0)$, $(1,0)$, and $(a,b)$ for integers $a,b$. Li--Paquin showed that the canonical transformation preserves, amongst other properties, the collinearity (or lack thereof) of the interior points, the number of interior points, and the number of boundary points. Thus, to show that $k$ is a $B4$-collinear integer, it suffices to show that all canonical lattice triangles with $k$ interior lattice points have all of their interior points collinear (and conversely also). 

\begin{lemma}
A lattice triangle with $4$ boundary points has all of its interior points collinear if and only if its canonical triangle has all of its interior points collinear. 
\end{lemma}

As an immediate corollary, 
\begin{theorem}
\label{canonical:label}
An integer $k$ is $B4$-collinear if and only if all canonical lattice triangles with $k$ interior points and $4$ interior points have all of their interior points collinear. 
\end{theorem}

We derive some results about canonical triangles. 

The area of a canonical triangle with free vertex $(a,b)$ and $k$ interior points is exactly half of $b$, as the base is of length $1$ and the height is $b$. (By reflecting across the $x$-axis if necessary, assume $b\ge0$.) By Pick's theorem, we additionally have that the area is equal to $k+\frac42-1=k+1$; it follows that $b=2k+2$. 

\begin{theorem}
\label{seglatpntctr:label}
Let $(x_1,y_1)$ and $(x_2,y_2)$ be distinct lattice points so that $x_1,y_1,x_2,y_2\in\zz$. There are exactly $\gcd(x_1-x_2,y_1-y_2)-1$ lattice points on a line segment from $(x_1,y_1)$ to $(x_2,y_2)$, excluding the endpoints. 
\end{theorem}

\begin{proof}
The line segment from $(x_1,y_1)$ to $(x_2,y_2)$ can be parametrized as $(x_1+(x_2-x_1)t,y_1+(y_2-y_1)t)$ for $t\in[0,1]$. This is a lattice point if and only if $(x_2-x_1)t$ and $(y_2-y_1)t$ are both integers, and there are exactly $\gcd(x_2-x_1,y_2-y_1)+1$ such $t$: they are \[0=\frac0{\gcd(x_2-x_1,y_2-y_1)},\frac1{\gcd(x_2-x_1,y_2-y_1)},\dots,\frac{\gcd(x_2-x_1,y_2-y_1)}{\gcd(x_2-x_1,y_2-y_1)}=1.\] Excluding the endpoints, which correspond to $t=0$ and $t=1$, we have $\gcd(x_2-x_1,y_2-y_1)-1$ points. 
\end{proof}

\begin{definition}
Fix a number of interior points $k$; this automatically fixes $b$ also. A choice of $a$ is \textit{valid} when $\gcd(a,b)$ and $\gcd(a-1,b)$ are $1$ and $2$ in some order (i.e. $\{\gcd(a,b),\gcd(a-1,b)\}=\{1,2\}$). 
\end{definition}

\begin{theorem}
The above definition is equivalent to the canonical lattice triangle with vertices at $(0,0)$, $(1,0)$, and $(a,b)$ having four boundary points. 
\end{theorem}
\begin{proof}
Having exactly four boundary points is equivalent to having two primitive edges and one non-primitive edge; in our canonical triangle, the bottom edge is fixed as primitive, so one of the two other edges is primitive and the other has one lattice point on it (aside from its endpoints). By \Cref{seglatpntctr:label}, as these segments have horizontal and vertical displacements $a,b$ and $a-1,b$, a choice of $a$ being valid is equivalent to having $\gcd(a,b)-1=0$ and $\gcd(a-1,b)-1=1$ or vice versa, i.e. $\{\gcd(a,b),\gcd(a-1,b)\}=\{1,2\}$. 
\end{proof}
\newpage
\subsection{The $k=5$ case}

\begin{theorem}
\label{five:label}
The number $5$ is $B4$-collinear. 
\end{theorem}

\begin{proof}
When $k=5$, the height $b$ must be $12$. 

By definition, $a$ is valid iff either $\gcd(a,12)=2$ and $\gcd(a-1,12)=1$, or $\gcd(a,12)=1$ and $\gcd(a-1,12)=2$. 

\textbf{Case 1.} In the former case, $a$ is even and $\gcd(a/2,6)=\gcd(a-1,3)=1$, so an even $a$ is valid iff $\gcd(a/2,6)=\gcd(a-1,3)=1$. Reducing modulo $12$, an even $a$ is in $\{0,2,4,6,8,10\}$; those satisfying $\gcd(a/2,6)=1$ are $2$ and $10$. Additionally, requiring $\gcd(a-1,3)=1$ shows that $a$ belongs to this case if and only if $a\equiv2\pmod{12}$. 

\textbf{Case 2.} In the latter case, $a$ is odd and $\gcd(a,3)=\gcd((a-1)/2,6)=1$. Thus, an odd $a$ is valid iff $\gcd(a,3)=\gcd((a-1)/2,6)=1$. Reducing modulo $12$, an odd $a$ is in $\{1,3,5,7,9,11\}$; those satisfying $\gcd(a,3)=1$ are $\{1,5,7,11\}$, and the only one $\gcd((a-1)/2,6)=1$ is $11$. 

Now, we have showed that the first case holds if and only if $a\equiv2\pmod{12}$, and the second case holds if and only if $a\equiv11\pmod{12}$. Recalling that $a$ is valid iff either the first or second case holds, it follows that $a$ is valid iff \[(a\bmod12)\in\{2,11\}.\] 

Note that the single non-vertex edge lattice point is at the midpoint of one of the edges; when $a\equiv2\pmod{12}$, it is at $(a/2,6)$, while when $a\equiv11\pmod{12}$ it is at $((a+1)/2,6)$. 

\textbf{Case 1.} Consider $a\equiv2\pmod{12}$; say $a=2+12m$ where $m$ is an integer. The open (i.e., excludes endpoints) line segment between $(1,0)$ and $(a/2,6)=(1+6m,6)$ lies entirely within the triangle's interior as it is a median of the triangle and contains exactly $5$ points as \[\gcd(1+6m-1,6-0)-1=6-1=5;\] these are all the interior points, and they are all collinear. 

\textbf{Case 2.} Consider $a\equiv11\pmod{12}$; say $a=11+12m$ where $m$ is an integer. The open median of the triangle from $(0,0)$ to $((a+1)/2,6)=(6+6m,6)$ lies entirely in the triangle and contains $5$ points as \[\gcd(6+6m,6)-1=6-1=5,\] which is all of the interior points - they are all collinear. 

Therefore, if $k=5$, all valid triangles have all of their interior points collinear. Then $5$ is $B4$-collinear. 
\end{proof}

\begin{theorem}
\label{smaller:label}
The $B4$-collinear integers between $1$ and $5$ are $\{1,2,5\}$. 
\end{theorem}

\begin{proof}
The integers $1$ and $2$ are trivially $B4$-collinear. As noted in the introduction, for any $k$, there exists a lattice triangle with $4$ boundary points and $k$ interior points; a singleton interior point is collinear, as is a set of two interior points. 

By \Cref{threeandfour:sample}, $3$ and $4$ are not $B4$-collinear. 

By \Cref{five:label}, $5$ is $B4$-collinear. 
\end{proof}

We will now prove that there are no $B4$-collinear integers greater than $5$. 

\section{A Non-Collinearity Criterion for $B(T)=4$}

\begin{theorem}
\label{noncolcrit:label}
Suppose $k\ge4$ and $3\le a\le k$ for some valid $a$. Then the interior points of such a triangle are not all collinear. 
\end{theorem}

\begin{remark}This indicates that for $k\ge4$, if there are any valid $a$ between $3$ and $k$, $k$ cannot possibly be $B4$-collinear - a sufficient criterion for being not $B4$-collinear. \end{remark}

\begin{proof}
We consider the number of lattice interior points of the canonical triangle lying on the line $x=1$. 


The intersection of the canonical triangle and the line $x=1$ is the line segment from $(1,0)$ to $\left(1,\frac ba\right)$; the number of interior lattice points on this segment is \[\left\lceil\frac ba\right\rceil-1=\left\lfloor\frac{b-1}a\right\rfloor=\left\lfloor\frac{2k+1}a\right\rfloor.\] We derive two inequalities. 

As $a\ge3$, \begin{equation}\left\lfloor\frac{2k+1}a\right\rfloor\le\left\lfloor\frac{2k+1}3\right\rfloor\le\frac{2k+1}3\le k-1,\end{equation} the latter inequality following when $k\ge4$. 

As $a\le k$, \begin{equation}\left\lfloor\frac{2k+1}a\right\rfloor\ge\left\lfloor\frac{2k+1}k\right\rfloor\ge\left\lfloor\frac{2k}k\right\rfloor=2.\end{equation} 

Putting these two together, we find that at least $2$ interior lattice points lie on the line $x=1$, but at least one interior point lies off of it. For all of the interior points to be collinear, there must exist a line incident to all of the points; any such line must be incident to the two points known to be on $x=1$. However, as two points fix a line, the only line passing through both is the line $x=1$ itself, which is not incident to the point off of $x=1$. This shows that the interior points are \textit{not} all collinear. 
\end{proof}

It follows that for $k\ge4$, if there are any valid $a$ between $3$ and $k$, not all interior lattice points will be collinear. We proceed to establish criteria for there being any valid $a$ between $3$ and $k$. 

\subsection{Alder's $\phi(m,2)$ function}

\begin{definition}
We define $g(k)$ to be the number of $0\le r<k$ such that $\gcd(k,r)=\gcd(k,2r+1)=1$ so that $g(k+1)$ is the number of $0\le r\le k$ such that $\gcd(r,k+1)=\gcd(2r+1,k+1)=1$. \[g(k+1):=|\{0\le r\le k\mid\gcd(r,k+1)=\gcd(2r+1,k+1)=1\}|\]
\end{definition}

\begin{remark} This is Alder's extension of the Euler totient function\cite{alder1958generalization} with the second argument fixed at $2$. \end{remark}

Before we continue, we establish certain properties about $g$. 

\begin{lemma}
The function $g$ is multiplicative. 
\end{lemma}

\begin{proof}
Suppose $m=m_1m_2$ where $\gcd(m_1,m_2)=1$. Our conditions for counting some $r$ in $g$ can be decomposed: $\gcd(r,m)=1\iff\gcd(r,m_1)=\gcd(r,m_2)=1$ and $\gcd(2r+1,m)=1\iff\gcd(2r+1,m_1)=\gcd(2r+1,m_2)=1$. Hence $0\le r<m$ is counted in $g$ if and only if $r\mod m_1$ and $r\mod m_2$ are counted by $g(m_1)$ and $g(m_2)$, respectively; by the Chinese Remainder Theorem, the decomposition $t\mapsto(t\mod m_1,t\mod m_2)$ is a bijection if we view it as a map from $\zz/m\zz$ to $\zz/m_1\zz\times\zz/m_2\zz$. It follows that $g(m)=g(m_1)g(m_2)$. 
\end{proof}

\begin{lemma}
If $p$ is prime and $c$ is a positive integer, \[g(p^c)=\begin{cases}p^{c-1}&p=2\\(p-2)p^{c-1}&p\neq2\end{cases}.\]
\end{lemma}

\begin{proof}
Suppose $p$ is a prime and $c$ is a positive integer. Then to count $g(p^c)$, we count the $0\le r<p^c$ which satisfy $\gcd(r,p^c)=\gcd(2r+1,p^c)=1$. Any $r$ which do not satisfy these constraints are those such that $p\mid r$ or $p\mid2r+1$. If $p=2$, the latter is immediate as $2r+1$ is odd, so the only bad residues are even ones; the count is $2^c-2^{c-1}=2^{c-1}=p^{c-1}$. If $p$ is an odd prime, $p>2$, and there are $p^{c-1}$ residues which are multiples of $p$. As $2$ is invertibel mod $p^c$, the congruence $2r+1\equiv1\pmod p$ has exactly one solution modulo $p$ and hence exactly $p^{c-1}$ solutions modulo $p^c$. Furthermore, these two counts are disjoint: $r$ and $2r+1$ are relatively prime, so $p$ cannot divide both of them. It follows that $g$ counts exactly $p^c-2p^{c-1}=(p-2)p^{c-1}$ numbers. 
\end{proof}

\subsection{A link between $g$ and $B4$-collinear integers}

We show two facts, assuming $k\ge4$: first, that there are exactly $g(k+1)-1$ valid $a$ between $3$ and $k$, and that $g(k+1)=1\iff k=5$. The latter will follow from the properties established above. 

\begin{theorem}
\label{former:label}
The number of valid $a$ between $3$ and $k$ is equal to $g(k+1)-1$. 
\end{theorem}

\begin{proof}
Recall that $a$ is valid when $\gcd(a,2k+2)=1$ and $\gcd(a-1,2k+2)=2$ or $\gcd(a,2k+2)=2$ and $\gcd(a-1,2k+2)=1$. These cases are disjoint: in the former case, $a$ must be odd, while in the latter case, $a$ must be even; thus, the number of valid $a$ is the sum of the number of valid $a$ in either case. 

In the former case, let $a=2r_1+1$; then the conditions are equivalent to \[\gcd(2r_1+1,k+1)=\gcd(r_1,k+1)=1.\] The latter case is equivalent to $\gcd(r'_2,k+1)=\gcd(2r'_2-1,k+1)=1$ with $a=2r'_2$. By properties of the greatest common divisor, $r_2:=k+1-r'_2$ satisfies the following: \[\gcd(r_2,k+1)=\gcd(2r_2+1,k+1)=1.\] There is a bijection between $r'_2$ and the corresponding $r_2$ value, so we can count $r_2$ instead. 

We now establish bounds on each $r_i$. We have $r_1:=\frac{a-1}2$, $r'_2:=\frac a2$, and $r_2:=k+1-\frac a2$. From $3\le a\le k$, the bounds are \[1\le r_1\le\left\lfloor\frac{k-1}2\right\rfloor,\qquad\left\lceil\frac k2+1\right\rceil\le r_2\le k-1.\] These ranges are disjoint and cover the integers in $[1,k-1]$ except one integer in the middle, which is $\lceil k/2\rceil$: $\frac k2$ if $k$ is even and $\frac{k+1}2$ if $k$ is odd. 

Say that $r$ is \textit{$k$-good} when $\gcd(r,k+1)=\gcd(2r+1,k+1)=1$. We want the number of $k$-good $1\le r\le k-1$ except for $\lceil k/2\rceil$. It is apparent that $g(k+1)$ counts the number of $k$-good $0\le r\le k$. 

We have $3$ numbers left to consider which are extraneously considered by $g$: \[[0,k]\setminus\left[1,\left\lfloor\frac{k-1}2\right\rfloor\right]\cup\left[\left\lceil\frac k2+1\right\rceil,k-1\right]=\left\{0,\left\lceil\frac k2\right\rceil,k\right\}.\] (Intervals are taken over the integers.) 

First, $0$ and $\lceil k/2\rceil$ are not $k$-good for any $k$, so these two are not a problem: $\gcd(0,k+1)=k+1>1$ for $0$; and if $k$ is even, \[\gcd(2\lceil k/2\rceil+1,k+1)=\gcd(k+1,k+1)=k+1>1\] while if $k$ is odd, \[\gcd(\lceil k/2\rceil,k+1)=\gcd((k+1)/2,k+1)=(k+1)/2>1.\] 

Now for $r=k$: this is considered by $g$ but does not correspond to a valid choice of $a$. We see that $\gcd(k,k+1)=\gcd(2k+1,k+1)=1$ for any $k$, so $k$ is $k$-good. Therefore, $g$ always overcounts by one, so the count we desire is exactly $g(k+1)-1$. We have demonstrated that the number of valid $a$ between $3$ and $k$ is $g(k+1)-1$. 
\end{proof}

We now consider the implications of $g(k+1)-1=0$, or $g(k+1)=1$. 

\begin{theorem}
\label{latter:label}
If $k\ge4$, $g(k+1)=1$ if and only if $k=5$. 
\end{theorem}

\begin{proof}
Recalling our properties above, for $p=2$, $g(2^e)=1$ if and only if $2^{e-1}=1$ or $e=0$, i.e. $e=1$ or $e=0$. For $p>2$, we have $g(p^e)=1$ if and only if $e=0$ or $(p-2)p^{e-1}=1$; as the product terms are positive integers, both must be $1$, i.e. $p=3$ and $e=1$ (or $e=0$). Thus, as $g(k+1)=1$ means that all product terms are $1$, $k+1$ is either $2^03^0=1$, $2^13^0=2$, $2^03^1=3$, or $2^13^1=6$. 

This gives $1$, $2$, $3$, and $6$, i.e. $k\in\{0,1,2,5\}$. We assumed $k\ge4$, so $g(k+1)=1\iff k=5$. 
\end{proof}

We have shown that if $k\ge4$ and $k\neq5$, then there is a valid $a$ which gives non-collinear interior points. Additionally, we showed earlier that $4$ is not $B4$-collinear, so it follows that no $k$ greater than $5$ is $B4$-collinear. This will be formalized in the next section. 

\section{$B(T)=4$ Assembly}
We revisit \Cref{main:label}. 

\textbf{Theorem 1.} The only $B4$-collinear integers are $1$, $2$, and $5$. 

\begin{proof}
In the following, assume that $k\ge4$. 

By \Cref{latter:label}, $g(k+1)-1=0$ iff $k=5$; applying \Cref{former:label} implies that there are no valid $a$ between $3$ and $k$ iff $k=5$. In other words, there is some valid $a$ between $3$ and $k$ iff $k\neq5$. By \Cref{noncolcrit:label}, if there is some valid $a$ between $3$ and $k$, the interior points of the respective triangle are not collinear, so $k$ is not $B4$-collinear. Thus, there are no $B4$-collinear integers greater than $5$. 

Combining this with \Cref{smaller:label}, the only $B4$-collinear integers are $\{1,2,5\}$. 
\end{proof}

\begin{remark}
We also performed extensive computational verification. For $k=5$, we exhaustively checked every valid canonical triangle with free vertex $(a,12)$ for $2 \le a \le 50,\!000$, and in every case the five interior lattice points were collinear. Furthermore, for every integer $6 \le k \le 10^6$, we computationally found a lattice triangle with $4$ boundary points whose interior lattice points are not collinear. These computations are consistent with the theoretical classification established in \Cref{main:label} and are not part of the proof.
\end{remark}

\section{$B(T)=5$ case}

First, consider any lattice triangle with $5$ boundary points. It is evident that it must have at least one primitive edge\footnote{Otherwise, there would be at least one non-vertex boundary point on each of the $3$ edges in addition to the $3$ vertices, leading to at least $6$ boundary points.}. Fix one of the vertices incident to a primitive edge to be the origin and say that the other vertices are $(x,y)$ and $(a',b')$, where the line segment from the origin to the former is primitive. By \Cref{seglatpntctr:label}, $\gcd(x,y)$ must be $1$. 

As in \Cref{canonical_section}, we can apply a transformation which sends $(x,y)$ to $(1,0)$ and preserves the number of interior points, the number of boundary lattice points, and the collinearity of any points. Denote the image of the other vertex $(a',b')$ under this transformation by $(a,b)$; this gives a canonical triangle with vertices at $(0,0)$, $(1,0)$, and $(a,b)$ with $5$ boundary points. The interior points of a lattice triangle with $5$ boundary points are collinear if and only if the interior points of the corresponding canonical lattice triangle are collinear. 

We now show that for every number of interior points $k$, no canonical triangle with $5$ boundary points and $k$ interior points has all of its interior points collinear. 

\begin{theorem}
Every canonical lattice triangle with $5$ boundary points has two primitive edges and one edge with two non-vertex lattice points. 
\end{theorem}

\begin{remark}This is true of all lattice triangles with $5$ boundary points, but we need it only for canonical lattice triangles.\end{remark}

\begin{proof}
Say our vertices are at $(0,0)$, $(1,0)$, and $(a,b)$. As $3$ of the $5$ boundary points are taken by the vertices, we consider the distribution of the two extra boundary points between the two remaining edges. 

There are two possibilities: one boundary point on each, or both to a single edge. Suppose our triangle had one primitive edge and two edges with one non-vertex boundary point each. By \Cref{seglatpntctr:label}, we have $\gcd(a,b)=\gcd(a-1,b)=2$, implying that $a$, $b$, and $a-1$ are all even. However, exactly one of $a$ and $a-1$ is even while the other must be odd - a contradiction. 

It follows that all canonical lattice triangles with $5$ boundary points have two primitive edges and one edge with two non-vertex boundary points on it. 
\end{proof}

As the edge between $(0,0)$ and $(1,0)$ is primitive, there is still the matter of which of the two other edges (between $(0,0)$ and $(a,b)$ and between $(1,0)$ and $(a,b)$) is primitive and which is not. By \Cref{seglatpntctr:label}, these two possibilities are captured by the condition $\{\gcd(a,b),\gcd(a-1,b)\}=\{1,3\}$. Hence, say a choice of $a$ is \textit{valid} when $\{\gcd(a,b),\gcd(a-1,b)\}=\{1,3\}$. This is equivalent to there being $5$ boundary lattice points on the canonical lattice triangle with vertices $(0,0)$, $(1,0)$, and $(a,b)$. 

Observe that for a valid choice of $a$, as either $\gcd(a,b)$ or $\gcd(a-1,b)$ is $3$, $3$ must divide $b$. 

Additionally, by Pick's theorem, we have that the area of such a triangle is $k+5/2-1=k+3/2$. By the standard triangle area formula, the area is $b/2$; it follows that $b/2=k+3/2$, so $b=2k+3$. 

\begin{theorem}
\label{notonetwo:label}
The number of interior points of any lattice triangle with $5$ boundary points is a multiple of $3$. 
\end{theorem}

\begin{proof}
It follows from $b=2k+3$ that a necessary condition of $a$ being valid is that $3\mid k$, as otherwise $3$ would not divide $2k+3=b$. Thus, if $3\nmid k$, then no integer $a$ is valid - i.e., the canonical triangle with free vertex at $(a,2k+3)$ for each $a$ does not have exactly $5$ boundary points. 

If, for the sake of contradiction, there were some lattice triangle with $5$ boundary points and $k$ interior points where $3\nmid k$, then its corresponding canonical triangle would have $5$ boundary points, $k$ interior points, and free vertex at $(a,2k+3)$ for some $a$; however, $3\nmid k\implies3\nmid2k+3$, meaning that our canonical triangle cannot possibly have $5$ boundary points - a contradiction. 

Therefore, for a given $k$ such that $3\nmid k$, there are no lattice triangles with $5$ boundary points and $k$ interior points. 

\end{proof}

\begin{corollary}
No integers which are not multiples of $3$ are $B5$-collinear. 
\end{corollary}

\begin{proof}
Our definition of $B5$-collinear integers specifically excludes vacuously satisfactory integers $k$ such that no triangles with $5$ boundary points and $k$ interior points exist. It follows that if $3\nmid k$, then $k$ is not $B5$-collinear. 
\end{proof}

\subsection{A Non-Collinearity Criterion for $B(T)=5$}

We investigate which multiples of $3$ are $B5$-collinear. Henceforth fix an arbitrary positive $k$ to be a multiple of $3$.\footnote{Vacuously, for any number of boundary points $n$, a triangle with $k=0$ interior points has all of its interior points collinear; since such a triangle always exists, we discount this trivial case in the definition, taking $k$ to be positive.} Additionally, fix $a=4$. This choice of $a$ is valid: $\gcd(4,2k+3)$ must be $1$ as the latter term is odd, and $\gcd(4-1,2k+3)=3$ as the latter is a multiple of $3$. Thus, this gives a canonical lattice triangle with vertices $(0,0)$, $(1,0)$, and $(4,b)=(4,2k+3)$; denote this triangle $T_k$. 

\begin{definition}
Let $T_k$ denote the lattice triangle with vertices at $(0,0)$, $(1,0)$, and $(4,2k+3)$. As reasoned above, it has exactly $5$ boundary points and $k$ interior points. 
\end{definition}

\begin{lemma}
There are exactly $\left\lfloor\frac{k+1}2\right\rfloor$ interior points of $T_k$ on the line $x=1$. 
\end{lemma}

\begin{proof}
The line $x=1$ intersects $T_k$ in two places: at $(1,0)$ and $(1,b/4)=(1,(2k+3)/4)$. By counting, there are exactly \[\left\lceil\frac{2k+3}4\right\rceil-1=\left\lceil\frac{2k-1}4\right\rceil=\left\lfloor\frac{2k+2}4\right\rfloor=\left\lfloor\frac{k+1}2\right\rfloor\] lattice points on the segment between $(1,0)$ and $(1,b/4)$, excluding the endpoints. 
\end{proof}

When $k\ge3$, we have $2\le\frac{k+1}2\le k-1$ and thus \begin{equation}2\le\left\lfloor\frac{k+1}2\right\rfloor\le k-1.\end{equation}

\begin{theorem}
\label{tknocol:label}
The interior points of $T_k$ are not collinear. 
\end{theorem}

\begin{proof}
As $k>0$ and we assume $k$ to be a multiple of $3$, $k$ must be at least $3$. By (3), it follows that there are at least two interior points on $x=1$ and at least one interior point off of $x=1$. As in \Cref{noncolcrit:label}'s proof, the interior points are not collinear as a line passing through the two interior points on $x=1$ necessarily does not pass through the interior point off of $x=1$. 
\end{proof}

\begin{theorem}
For any positive integer $k$, either there is a lattice triangle with $k$ interior points and $5$ boundary points whose interior points are not collinear, or there are no lattice triangles with $k$ interior points and $5$ boundary points at all. 
\end{theorem}

\begin{proof}
Given a fixed integer $k>0$, if $k$ is a multiple of $3$, the lattice triangle $T_k$ has $k$ interior points and $5$ boundary points, but its interior points are not collinear by \Cref{tknocol:label}. If $k$ is not a multiple of $3$, no triangles exist at all by \Cref{notonetwo:label}. 
\end{proof}

It immediately follows that there are no $B5$-collinear numbers, and we have shown \Cref{main5:label}. $\Box$

\section{Conclusion}
We have completely classified the integers that force the interior lattice points of a lattice triangle to be collinear in the cases $B(T)=4$ and $B(T)=5$. We proved that the only $B4$-collinear integers are $1, 2,$ and $5$, whereas \textit{no} $B5$-collinear integers exist. Together with the classification by Li and Paquin for $B(T)=3$, these results provide a complete picture for the first three non-trivial boundary cases.

The classifications reveal a contrast between four and five boundary lattice points. While four boundary lattice points force collinearity for a single nontrivial interior-point count, five boundary points provide geometric flexibility that prevents forced collinearity for any $k$. Whether this marks the beginning of a more general pattern remains an intriguing question.

Several natural questions arise for further research. The next case is $B(T)=6$: which integers, if any, are $B6$-collinear? More generally, can the $Bn$-collinear integers be classified for arbitrary $n$? Does a general pattern emerge as the number of boundary lattice points increases, particularly as $n \to \infty$? More broadly, can one characterize when the boundary structure of a lattice polygon beyond triangles forces rigid geometric configurations, such as collinearity or specific convex hull structures, of its interior lattice points? We hope that the methods developed here provide a useful foundation and inspiration for investigating these questions. 

\bibliographystyle{plain}
\bibliography{refs}

\end{document}